\documentclass[letterpaper,10pt]{article}

\usepackage[margin=1.5in]{geometry} 
\usepackage[toc,page]{appendix}
\usepackage{amssymb,amsmath,latexsym,amsthm,amsfonts}
\usepackage{mathrsfs} 
\usepackage[pdftex,pagebackref=true,colorlinks=true,linkcolor=blue,unicode]{hyperref} 

\usepackage{authblk} 

\usepackage[toc,page]{appendix}

\usepackage{subcaption}
\usepackage[normalem]{ulem}
\usepackage[bottom]{footmisc}
\usepackage{xcolor}
\usepackage{graphicx}
\newcommand{\I}[1]{ {\bf 1}_{\left\{ #1 \right\}} }
\newcommand{\R}{\mathbb{R}}

\newcommand{\EE}{\mathbb{E}}
\newcommand{\OO}{O}
\newcommand{\PP}{ \mathbb{P}}

\newcommand{\B}[2]{ \mathrm{b} {( #1, \, #2 )} }
\newcommand{\condS}{(S_s)_ {s \in \mathbb{R}_+}}

\newtheorem{theorem}{Theorem}
\newtheorem{lemma}{Lemma}
\newtheorem{proposition}{Proposition}

\newtheorem{definition}{Definition}

\title{Limit distributions of the upper order statistics for the L\'evy-frailty Marshall-Olkin distribution}
\author[1]{Javiera Barrera}
\author[1,2]{Guido Lagos}
\affil[1]{Faculty of Engineering and Sciences, Universidad Adolfo Ib\'a\~nez}
\affil[2]{Industrial Engineering Department, University of Santiago}
\date{\today}

\begin{document}

\maketitle

\begin{abstract}

The Marshall-Olkin (MO) distribution has been considered a key model in reliability theory and in risk analysis, where it is used to model the lifetimes of dependent components or entities of a system and dependency is induced by ``shocks'' that hit one or more components at a time.
Of particular interest is the L\'evy-frailty subfamily of the Marshall-Olkin (LFMO) distribution, since it has few parameters and because the nontrivial dependency structure is driven by an underlying L\'evy subordinator process.
The main contribution of our work is that we derive the precise asymptotic behavior of the upper order statistics of the LFMO distribution.
More specifically, we consider a sequence of $n$ univariate random variables jointly distributed as a multivariate LFMO distribution and analyze the order statistics of the sequence as $n$ grows.
Our main result states that if the underlying L\'evy subordinator is in the normal domain of attraction of a stable distribution with index of stability $\alpha$ then, after certain logarithmic centering and scaling, the upper order statistics converge in distribution to a stable distribution if $\alpha>1$ or a simple transformation of it if $\alpha\leq1$.
Our result is especially useful in network reliability and systemic risk, when modeling the lifetimes of components in a system using the LFMO distribution, as it allows to understand the behavior of systems that rely on its last working components.
Our result can also give easily computable confidence intervals for these components, provided that a proper convergence analysis is carried out first.

\end{abstract}


\section{Introduction}

In this paper we explore the upper order statistics of the \emph{L\'evy-frailty} (LFMO) subfamily of the \emph{multivariate Marshall-Olkin (MO) exponential distribution}, also known as the  \emph{conditionally-iid} construction of the MO distribution.
A (general) MO distribution is a positive random vector in $\R^n$ that can be intuitively understood as if each component of the vector represents the ``lifetime'' of a component in an $n$-component system subject to failures as follows.
At time zero, all components are working; there are ``shocks'' that occur after independent exponentially distributed times; each shock hits a subset of components, simultaneously killing all the components in the subset that were not already killed; and killed components stay in that state forever after.

The study of the MO distribution is deeply connected with the copula literature.
Copulas are multivariate random variables whose marginals are uniformly distributed; their popularity relies on Sklar's theorem, that allows to decompose any multivariate distribution with continuous marginals into a unique set of copula and marginal distributions, thus separating the multivariate distribution into its dependency structure ---in the copula--- and its marginal distributions.
From this perspective the MO multivariate distribution is obtained by pairing, using the survival copula version of Sklar's theorem, the \emph{Marshall-Olkin copula} together with exponential marginal distributions; see~\cite[p.~33]{matthias2017simulating} for further details.
The \emph{Marshall-Olkin copula} has remarkable applications in extreme-value theory; see~\cite[p.~119]{matthias2017simulating} and references therein. For instance, the copula is used to capture the dependency structure in rare event settings ---see~\cite{gudendorf2010extreme}--- since it satisfies the extreme-value property for copulas; see~\cite[p.~285ff]{matthias2017simulating}.
More generally both the MO copula and the multivariate distribution are now considered key tools in reliability theory and quantitative risk management; see~\cite{engel2017one,lindskog2003common}.

In turn, the L\'evy-frailty Marshall-Olkin (LFMO) distribution is a particular case of the MO distribution that alleviates the high number of parameters needed to specify a ``crude'' MO distribution.
This is done by imposing exchangeability between the components' lifetimes by means of introducing a L\'evy subordinator process that acts as a latent factor.
Although this distribution was proposed in~\cite{mai2009levy} using a construction that introduces dependence to an initially independent collection of exponential random variables, it was only established afterwards in~\cite{mai2011reparameterizing} that the proposed model corresponds to a subfamily of the MO distribution.
As a result, the LFMO distribution is a flexible and powerful modeling tool that requires few parameters and is easy to simulate as long as the L\'evy subordinator is as well; see~\cite[p.~150]{matthias2017simulating}.

Our interest in the upper order statistics of the LFMO distribution is threefold.
First, we are motivated by the rich combinatorial structure of the order statistics of the LFMO distribution.
As we discuss in \autoref{sec:preliminaries}, the first and average order statistics are relatively easy to address; however, the upper order statistics exhibit a much more complex behavior that is difficult to grasp analytically.
Second, we are also motivated by the use of order statistics in the field of network reliability when describing the probabilistic behavior of a system's lifetime.
In particular, using the so-called \emph{Samaniego signature} result, for a large family of systems one can express the probability of a system being working in terms of the tail probability of the order statistics of the component's failure times; see~\cite{marichal2011signature}.
Third and last, the study of the stochastic behavior of extremal order statistics is a classical topic in statistics and probability theory.
Indeed, these quantities are essential tools in statistical inference, where they have a centuries-old history of applications, see, e.g.,~\cite{david2006topics}; and are also the essential object of study in extreme-value theory; see, e.g.,~\cite{finkenstadt2003extreme}.
All in all, it is important to explore the possible regimes of the order statistics of the components of the LFMO distribution, given its wide application as a model for simultaneous failures of components in systems.

\paragraph{Main contributions}

The main contributions of our work are threefold.

First and foremost, we derive the precise asymptotic behavior of the upper order statistics of the LFMO distribution as the dimension of the space grows.
Indeed, we show that after certain logarithmic centering and scaling, all the upper order statistics converge in distribution to a nondegenerate limit random variable, and we give this limit distribution in an analytical closed-form expression in terms of a stable random variable.
Additionally, our results are particularly convenient for computational purposes: all the required constants are usually easy to compute, the limit random variables are also easy to simulate, and our results give a probabilistic upper bound on the time horizon required to simulate the underlying L\'evy subordinator of the LFMO distribution.
From this perspective, our result is also a contribution to a body of work in applied and computational probability concerned with providing computationally tractable approximations for the analysis of probabilistic tools and models; quintessential examples of this type are heavy-traffic limits in queues, fluid limits and diffusion approximations; see, e.g.,~\cite{asmussen2007stochastic,asmussen2003applied}.

Second, our main result gives important qualitative insight into the effects of choosing the underlying L\'evy subordinator involved in the LFMO distribution.
Indeed, our results hold under mild conditions, namely, that the L\'evy subordinator lies in the \emph{normal domain of attraction} of a stable distribution, and we show that there is a critical behavior change in terms of the index of stability $\alpha \in (0,2]$ of the limit stable random variable.
In detail, and using the perspective of the LFMO distribution as modeling failure times in a system with $n$ components, we show that $\alpha=1$ is a critical value since for $\alpha>1$, the instant when the last component fails grows as $\OO(\log n)$ and concentrates around that value.
This is qualitatively the same behavior as when the lifetimes of the components are iid.
On the other hand, when $\alpha \leq 1$, there is no longer a concentration phenomenon.

Third, in network reliability and systemic risk, our results allows to give estimations for the last failure times of the system's components when using the LFMO distribution to model the failure times of components; and even precise confidence interval can be given if proper convergence of the sequence has been established.
Moreover, our results can potentially lead to give qualitative insight on the probabilistic behavior of more generally defined system failure times when using the \emph{Samaniego signature} decomposition; see, e.g.,~\cite{marichal2011signature}.

\paragraph{Literature review}
The study of extreme values and order statistics has been historically motivated by the study of floods, droughts, fatigue failures and other engineering applications; see~\cite{david2006topics}.
This has become a classical field, and its most relevant theorem is the Fisher-Tipett-Gnedenko result, which gives the general asymptotic behavior for the maximum and minimum of an iid sample.
In the process of understanding the dependence structure of an exchangeable sequence of random variables  a natural step is to study the asymptotic behavior of its extremes values; this was done in~\cite{wuthrich2005limit} for a sequence whose dependency is induced by a mixture model, but until now it has not been done for the MO model. 

Regarding the MO distribution, its study has been driven mostly by its application to reliability and systemic risk modeling, especially in the last decade or so; see~\cite{cherubini2015marshall}.
This model was first proposed by~\cite{Marshall1967} as a means to generalize the lack of memory property to two-dimensional distributions, and it was later generalized to any finite dimension.
It is precisely the memoryless property that has popularized this distribution and its survival copula, as it allows efficient simulation techniques; see, e.g.,~\cite[Ch.~3]{matthias2017simulating} and~\cite{Botev2015} for rare event simulation.
Nonetheless, an important shortfall of the general MO model is that it requires to specify a number of parameters that is exponential in the number of components when no further assumptions or simplifications are made.
In some real-world situations, the nature of the problem can lead to natural simplifications that reduce the number of parameters, e.g.,~in portfolio-credit risks or in insurance; see~\cite{bernhart2015survey}.
In this line,~\cite{MatusIEEE} propose a Lasso selection model to obtain the parameters.
Another way to address the explosive number of parameters is to impose exchangeability of the components' lifetimes, as is the case of the LFMO model proposed by~\cite{mai2009levy} and that we will discuss in detail in~\autoref{sec:preliminaries}.

\paragraph{Notation}
Throughout the paper, we will use the following notation.
Given a collection of random variables $(\xi_n)_n$ and $\xi_\infty$, we write $\xi_n \Rightarrow \xi_\infty$ when the sequence $(\xi_n)_n$ converges in distribution to $\xi_\infty$, i.e., the distributions of the random variables $(\xi_n)_n$ converge weakly to the distribution of $\xi_\infty$.
Additionally, for a real random variable $\xi$, we denote by $\EE \xi$ its expected value and $\text{Var}(\xi)$ its variance.
As usual, we denote by $\R_+$ the nonnegative real numbers.
In addition, for two functions $f$ and $g$, we write $f(x) \sim g(x)$ when $\lim_{x \to +\infty} f(x)/g(x) = 1$; we write $f(x) = o \left( g(x) \right)$ when $\lim_{x \to +\infty} f(x)/g(x) = 0$; and we write $f(x) = \OO \left( g(x) \right)$ when  $\limsup_{x \to +\infty} f(x)/g(x) < +\infty$.
Lastly, for a set $A$ we denote as $\I{x \in A}$ the function in $x$ that is equal to one when $x \in A$ and zero otherwise.

\bigskip 

\paragraph{Organization of this paper}
In~\autoref{sec:preliminaries} we give the preliminary concepts that contextualize our work.
In~\autoref{sec:MainR}, we state our main result regarding the asymptotic behavior of the upper order statistics of a multivariate LFMO distribution.
Finally, in~\autoref{sec:exp}, we show simulation results that computationally test the convergence of our results presented in~\autoref{sec:MainR}.

\section{Preliminaries}\label{sec:preliminaries}

In this section we give an overview of the mathematical models and tools we consider in this paper.
For that, in \autoref{subsec:MO} we define the Marshall-Olkin (MO) distribution and the L\'evy-frailty (LFMO) sub-family, and in \autoref{subsec:order stats}, we give a broad analysis of the order statistics of the LFMO distribution. 
For more details about the applications and properties of the MO distribution we refer the reader to~\cite{cherubini2015marshall} and~\cite[Ch.~3]{matthias2017simulating}.

\subsection{Marshall-Olkin and L\'evy-frailty Marshall-Olkin distributions}\label{subsec:MO}

\begin{definition}[Marshall-Olkin distribution]\label{def:MO}
A random vector $T \in \R^n$ is said to have a Marshall-Olkin~(MO) distribution in $\R^n$ if its components $(T_i)_i$ are defined as
$$T_i= \min\{Z_V \ : \ V \subset \{1,\ldots,n \}~,~i \in V\}.$$
where $(Z_V)_V$  is a family of independent exponential random variables with rate $\lambda_V$ for each $V \subset \{1,\ldots,n \}$.
We assume that $\lambda_V \in \R_+$, with the convention that $Z_V=\infty$ if $\lambda_V= 0$, and additionally assume that $\max_{V : i \in V} \lambda_V > 0$ for each $i$.
\end{definition}
An intuitive interpretation of~\autoref{def:MO} consists of thinking of $T_i$ as the lifetime of the component $i$ in a system with components $\{ 1, \ldots, n \}$, each of which can be working or not working.
All components are initially working, and for each subset $V \subseteq \{1,\ldots,n\}$ at time $Z_V$, there is a ``shock'' that hits all components of $V$.
Once a shock hits a working component, it stops working and stays in that state forever.
In this way, the time $T_i$ at which a component $i$ stops working is the first time that a shock hits it.
We remark that all the $T_i$ are marginally exponentially distributed.
Additionally, the assumption that some shocks' rates could be zero is equivalent to assuming that those shocks will never occur, which is used in certain settings for modeling reasons.

\begin{definition}[L\'evy-Frailty Marshall-Olkin distribution]\label{def:LFMO}
A random vector $T$ in $\R^n$ is said to have a L\'evy-frailty Marshall-Olkin~(LFMO) distribution in $\R^n$ if its components $(T_i)_i$ are defined as
\begin{align}\label{eq:Ti}
T_i := \inf \left\{ t \geq 0 \ : \ S_t \geq \varepsilon_i \right\}, \qquad i = 1, \ldots, n,
\end{align}
where $S=(S_t)_{t \geq 0}$ is a L\'evy subordinator with $S_0 = 0$ and $\{ \varepsilon_1, \ldots, \varepsilon_n \}$ is a collection of iid exponential random variables, the ``triggers'',  with unit parameter and which are independent of $S$.
\end{definition}
An intuitive interpretation of~\autoref{def:LFMO} consists again of thinking of $T_i$ as being the lifetime of component $i$ in an $n$-component system in which all components are working at time zero but component $i$ fails the first time that the L\'evy subordinator $S$ upcrosses the trigger $\varepsilon_i$ associated with $i$.
This construction is equivalent in distribution to defining a MO distribution with the following shock rates $(\lambda_{V})_V$:
$$
\lambda_{V}= \sum_{i=0}^{|V|-1} {{|V| -1}\choose{i}} (-1)^i \left[ \psi(n-|V|+i+1) -\psi(n-|V| +i) \right],
 $$ 
for all  $V\subset \{1,\ldots,n\}$, where $\Psi(x) = -\log \EE e^{-x S_1}$ is the Laplace exponent of the L\'evy subordinator $S$; see, e.g.,~\cite{bernhart2015survey}.
Note that the latter parameters $\lambda_V$ only depend on $|V|$, thus in particular making the components \emph{exchangeable}; see \cite{mai2013sampling}.

We remark that the LFMO distribution can be viewed as the result of imposing on the MO distribution the existence of a ``latent'' stochastic variable such that, conditional on the value of it, the random variables $T_1, \ldots, T_n$ are iid.
 It turns out that the L\'evy subordinator $S$ characterizes such a latent variable, and conditional on the value of $S$, we have that $T_1, \ldots, T_n$ are iid with $\PP(T_i > t \, | \, S) = e^{-S_t}$ for all $t \geq 0$.
This outcome inspires the \emph{conditionally-iid} and \emph{L\'evy-frailty} terminology; see, e.g.,~\cite{mai2014multivariate}.
For the proof of the correspondence between the two constructions see Theorem 3.2 in~\cite[ch.~3]{matthias2017simulating}.

We also remark that the MO distribution in $n$ dimensions of~\autoref{def:MO} requires specifying a number of $2^n-1$ parameters $\lambda_V$, whereas the LFMO distribution only requires parameterizing the L\'evy subordinator $S$.
In this sense, the LFMO distribution is a way to reduce the parametric complexity of the MO distribution by means of introducing the latent variable $S$.
In fact, most characteristics of the LFMO model are expressed in terms of $S_1$ and its Laplace exponent $\Psi(x) = -\log \EE e^{-x S_1}$, which have, respectively, rich probabilistic and analytical structures that allow to study and exploit the model in great detail; see, e.g.,~\cite{bernhart2015survey,hering2012moment}.
Finally, it is worth mentioning that~\cite{engel2017one} studies the LFMO construction when the triggers are nonhomogeneous, that is, when the triggers in~\autoref{def:LFMO} are independent and exponential but with different rates.
For further details about the LFMO distribution see~\cite{mai2009levy,matthias2017simulating}.

\subsection{Order statistics of the LFMO distribution}\label{subsec:order stats}
The order statistics of the LFMO distribution are, in general, not easy to perform calculations with, despite having a rich combinatorial structure.
To illustrate this fact, we now give a straightforward result regarding the marginal distribution of the order statistics  of the LFMO model.
It is proved by applying R\'enyi's representation of the order statistics of iid exponential random variables ---see e.g.~\cite[eq.~(11.2)]{nagaraja2006order}--- together with the distribution of the sum of arbitrary independent exponential random variables; see~\cite[Lemma 11.3.1]{nagaraja2006order}.

\begin{proposition}
Let $T$ in $\R^n$ be an LFMO distributed random vector and denote by $T_{m:n}$ the $m$-th increasing order statistic of $T$, i.e., $\{T_{1:n}, \ldots, T_{n:n}\} = \{T_1, \ldots, T_n\}$ and $T_{1:n} \leq \ldots \leq T_{n:n}$.
For all $m=1,\ldots,n$ and $t \geq 0$, it holds that
 \begin{align}\label{eq:Ti tail}
\PP(T_{m:n} > t)
= \sum_{k=n-m+1}^n e^{-\psi(k) t} {n \choose k} {{k-1} \choose {n-m}} (-1)^{k-n+m-1},
 \end{align}
where $\Psi$ is the \emph{Laplace exponent} of $S_1$, i.e., $\Psi(x) = -\log \EE e^{-x S_1}$.
\end{proposition} 

From~\eqref{eq:Ti tail}, it follows that the first order statistic $T_{1:n}$ takes a simple form, as it is marginally distributed as an exponential random variable with mean $1/\Psi(n)$.
Additionally,~\cite[Lemma 3.3]{fernandez2015mean} gives a simple argument to obtain the limit distribution of the average order statistics.
In contrast, the upper order statistic $T_{n:n}$ is much more complex; for instance, its mean is
\begin{align}\label{eq:mean tail ord}
\EE T_{n:n} = \sum_{k=1}^n  {n \choose k} \frac{(-1)^{k-1}}{\psi(k)}.
\end{align}
Intuitively, though, if $\EE S_1 < +\infty$, then $\EE T_{n:n}$ should asymptotically behave as $(\log n + \gamma) / \EE S_1$, where $\gamma$ is the Euler-Mascheroni constant, since $\EE S_t = t \EE S_1$  and the mean of the maximum of $n$ iid exponential random variables of parameter 1 is $\sum_{k=1}^n 1/k \approx \log n + \gamma$.
On the other hand, if $\EE S_1 = +\infty$ and the L\'evy subordinator is stable ---so in particular $\alpha<1$--- then $\PP(S_{t_n} > (\log n+ \gamma) ) = \PP(S_1 > (\log n+ \gamma)/t_n^{1/\alpha} ) \to  1$ when $t_n$ is such that $t_n/(\log n)^{\alpha} \to +\infty$.
This result suggests that, for large $n$, with probability close to one, the trajectory of the subordinator will climb sufficiently fast that it surpasses all exponential triggers ---thus killing all components--- in an interval of size $o(\log n)$.

In perspective, the previous behaviors of the upper order statistics are not obvious to deduce from equation~\eqref{eq:mean tail ord}.
However, in our main result in the following section, we not only corroborate them but also give the precise asymptotic distribution of all the upper order statistics of the LFMO distribution.

\section{Main result}\label{sec:MainR}

In this section, we give the main result our paper, which shows the asymptotic behavior of the upper order statistics of the LFMO distribution.
We will consider the following two hypotheses.
\begin{description}
	\item[\bf\hypertarget{hypA}{Hypothesis~(A$_\alpha$)}] $\text{Var} (S_1) = +\infty$, i.e., the variance of $S_1$ is infinite, and $\PP(S_1 > t) \sim A / t^\alpha$ as $t \to +\infty$ for some $\alpha \in (0, 2)$ and $A>0$.
	\item[\bf\hypertarget{hypB}{Hypothesis~(B)}] $0 < \text{Var} (S_1) < +\infty$.
\end{description}
Additionally, when working under \hyperlink{hypA}{Hypothesis~(A$_\alpha$)}, we will use the constant $C_\alpha$ defined as
\begin{align}\label{def:C}
C_\alpha
:= \begin{cases}
\frac{1-\alpha}{\Gamma(2-\alpha) \cos(\pi \alpha/2)} & \text{if } \alpha \in (0,1) \cup (1,2) \\
2 / \pi & \text{if } \alpha =1.
\end{cases}
\end{align}
Finally, we will use the parameterization and notation of~\cite[Section 4.5.1]{whitt2002stochastic} for stable random variables.
The following is the main result of the paper.

\begin{theorem}\label{theo1}
Let $T$ be a random vector in $\R^n$ having a L\'evy-frailty Marshall-Olkin (LFMO) distribution in $n$ dimensions with underlying L\'evy subordinator $S$, as in~\autoref{def:LFMO}.
Denote by $T_{m:n}$ the $m$-th increasing order statistic of the components of $T$, i.e., $\{T_{1:n}, \ldots, T_{n:n}\} = \{T_1, \ldots, T_n\}$ and $T_{1:n} \leq \ldots \leq T_{n:n}$.

\begin{enumerate}

\item	Assume that $\EE S_1 < +\infty$ and either \hyperlink{hypA}{Hypothesis~(A$_\alpha$)} holds for some $\alpha \in (1,2)$ or \hyperlink{hypB}{Hypothesis~(B)} holds, in which case we set $\alpha := 2$.
Let $(m_n)_n$ be any integer sequence satisfying $1 \leq m_n \leq n$ for all $n$ and $n-m_n = o\left( (\log n)^{1/\alpha} \right)$ as $n \to +\infty$.
Then, as $n \to +\infty$, we have that
\begin{align}\label{theo1 lim1}
\frac{T_{m_n:n} - \log n / \EE S_1}{\left(\log n \right)^{1 / \alpha} / \EE S_1} \Longrightarrow \Sigma,
\end{align}
where $\Sigma$ is a $\text{Stable}_\alpha (\sigma, -1, 0)$ random variable under \hyperlink{hypA}{Hypothesis~(A$_\alpha$)} and a \linebreak $\text{Normal}(0,\sigma^2)$ random variable under \hyperlink{hypB}{Hypothesis~(B)}, with $\sigma$ defined as
\begin{align}\label{def:sigma}
\sigma := \begin{cases}
\left( \frac{A}{C_\alpha \EE S_1} \right)^{1 / \alpha} & \text{under \hyperlink{hypA}{Hypothesis~(A$_\alpha$)}} \\
\sqrt{\text{Var} (S_1) / \EE S_1} & \text{under \hyperlink{hypB}{Hypothesis~(B)}},
\end{cases}
\end{align}
and $C_\alpha$ defined as in~\eqref{def:C}.

\item Assume that $\EE S_1 = +\infty$ and \hyperlink{hypA}{Hypothesis~(A$_\alpha$)} holds for some $\alpha \in (0,1]$.
Let $(m_n)_n$ be an integer sequence satisfying $1 \leq m_n \leq n$ for all $n$ and $n-m_n = o\left( n^\rho \right)$ as $n \to +\infty$ for all $\rho \in (0,1)$.
Then, as $n \to +\infty$, it holds that
\begin{align}\label{theo1 lim2}
\frac{T_{m_n:n}}{\left( \log n \right)^{1 / \alpha}} \Longrightarrow \frac{1}{\Sigma^{\alpha}},
\end{align}
where $\Sigma$ is a $\text{Stable}_\alpha (\sigma, 1, 0)$ random variable with $\sigma := \left( A/C_\alpha\right)^{1 / \alpha}$ and $C_\alpha$ is as defined in~\eqref{def:C} above.
\end{enumerate}
\end{theorem}

The proof of the previous result is shown in~\autoref{sec:proof}.
The key idea is to exploit the conditionally iid property to write the distribution of the left-hand side of~\eqref{theo1 lim1}, conditional on the path of $S$, as a deterministic function evaluated in a certain ``zoom-out'' centering and scaling of $S$.
We then show that the thus obtained sequence of deterministic functions converges pointwise and that the ``zoom-out'' version of $S$ converges in distribution.
We conclude that both limits hold simultaneously by using a generalized version of the continuous mapping theorem for weak convergence.

\paragraph{Remarks on~\autoref{theo1}}
\begin{enumerate}
	\item First, we note that \hyperlink{hypA}{Hypotheses~(A$_\alpha$)} and~\hyperlink{hypB}{(B)} describe all the non-trivial cases in which the distribution of $S_1$ is in the \emph{normal domain of attraction} of a stable distribution; that is, the cases in which there exists a sequence $(\mu_n)_n$ and a $\rho>0$ such that $(S_n-\mu_n)/n^\rho$ converges in distribution as $n \to +\infty$ to a stable distribution.
	See, e.g.,~\cite[Section 4.5]{whitt2002stochastic} for further details.
	
	\item The case of $\text{Var} (S_1) = 0$ corresponds to the trivial case $S_t = S_1 t$ with $S_1$ a deterministic constant.
	In this case the components $T_i$ are iid distributed exponential random variables with mean $1/S_1$, and it is easily shown that, e.g., $S_1 \left( T_{n:n} - \log n / S_1 \right)$ converges in distribution to a standard $\text{Gumbel}$ distribution as $n \to +\infty$.

	\item When both $\EE S_1 = +\infty$ and \hyperlink{hypA}{Hypothesis~(A$_\alpha$)} hold, then necessarily, $\alpha \in (0,1]$.
On the other hand, when $S_1$ satisfies $\PP(S_1>t) \sim A/t^\alpha$ as $t \to +\infty$ for some positive $A$ and an $\alpha > 2$, then necessarily, $\text{Var}(S_1)<+\infty$, in which case \hyperlink{hypB}{Hypothesis~(B)} holds.

	\item The distribution of the random variable on the right-hand side of~\eqref{theo1 lim2} converges as $\alpha \searrow 0$ to an exponential distribution with mean $1/\sigma^\alpha$; see, e.g.,~\cite{maller2018small}.
Informally then, under the assumptions of part 2.~of \autoref{theo1}, for large $n$ and $\alpha$ close to zero, the distribution of $T_{m_n:n}$ is ``similar'' to an exponential distribution with mean $(\log n)^{1/\alpha}/ \sigma^\alpha$.

	\item Lastly we remark that one obtains a different asymptotic regime if instead one assumes that the components are actually independent instead of dependent --- recall that the components of $T$ are dependent and are all exponentially distributed with mean $1/\Psi(1)$, where $\Psi(1) = -\log \EE e^{-S_1}$.
	Indeed, if the components were instead assumed to be iid with the same marginals as before then now $\Psi(1) \left( T_{n:n} - \log n / \Psi(1) \right)$ converges in distribution, as $n \to +\infty$, to a standard $\text{Gumbel}$ distribution.
		
\end{enumerate}

\paragraph{Extensions.}

We conjecture that the behavior we established in our result may be observed in other multivariate models.
For instance, in the general MO distribution, we believe that if shocks affecting massive subsets of components occur frequently then we may obtain a behavior similar to the case in which $\alpha \in (0,1)$, i.e., as in part~1.~of \autoref{theo1}; and if massive shocks are infrequent then we may obtain a behavior similar to the case in which $\alpha>1$, i.e., as in part~2.~of~\autoref{theo1}.
Nonetheless, in the general MO setting a key challenge is to find the parameter or structure that generalizes the role of the index of stability $\alpha$ in the LFMO case.


Also, recently in~\cite{mai2018extreme} the author uses a construction similar to the LFMO one to obtain a min-stable multivariate exponential distribution that, moreover, is associated with remarkable copulas such as the Galambos and Gumbel copulas.
They obtain this distribution by replacing the L\'evy subordinator by another stochastic process.
Therefore, asymptotic results similar to ours may possibly be obtained by exploiting these similarities.

\bigskip

A direct application of  \autoref{theo1} is to obtain confidence intervals for the upper order statistics $T_{m_n : n}$ when $n$ is considered sufficiently large.
Nonetheless, for a confidence interval to be valid, one should take care of properly establishing that the number of components $n$ is large enough so that convergence has been attained. 
 This motivates the experiments and discussion we carry out in the following section.




\section{Computational experiments}\label{sec:exp}

To empirically test the convergence in~\autoref{theo1}, in this section we show the results of performing a large number of Monte Carlo simulations of the random variables on the left- and right-hand sides of~\eqref{theo1 lim1}.

\begin{figure}[h!]
	\centering

	\begin{subfigure}{0.8\textwidth}
		\centering
		\includegraphics[width=\textwidth]{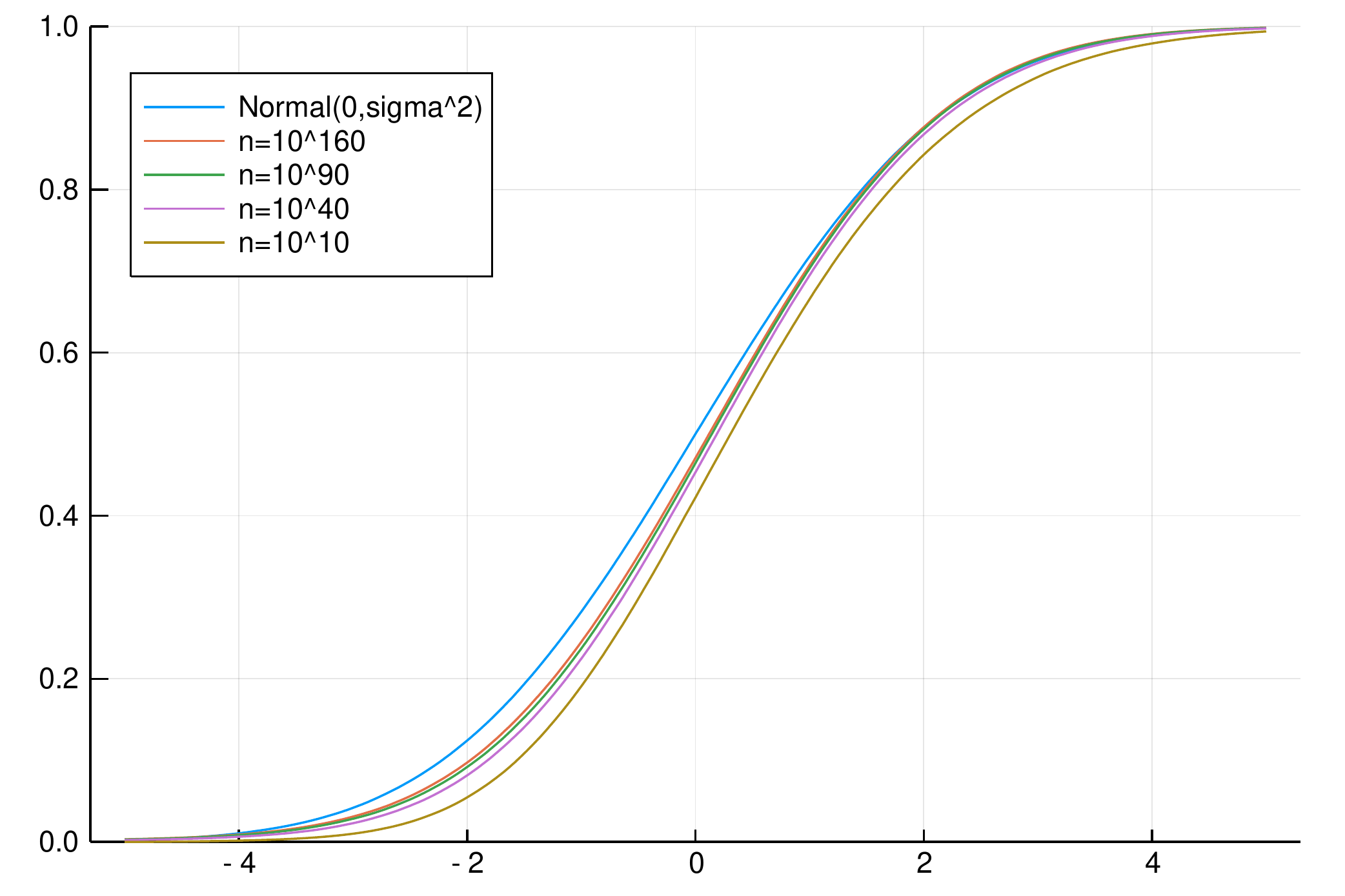}
		\caption{CPP($\lambda=1$) with Pareto($\alpha=2.5$) steps}
		\label{subfig:alpha200}
	\end{subfigure}%
	
	\begin{subfigure}{0.8\textwidth}
		\centering
		\includegraphics[width=\textwidth]{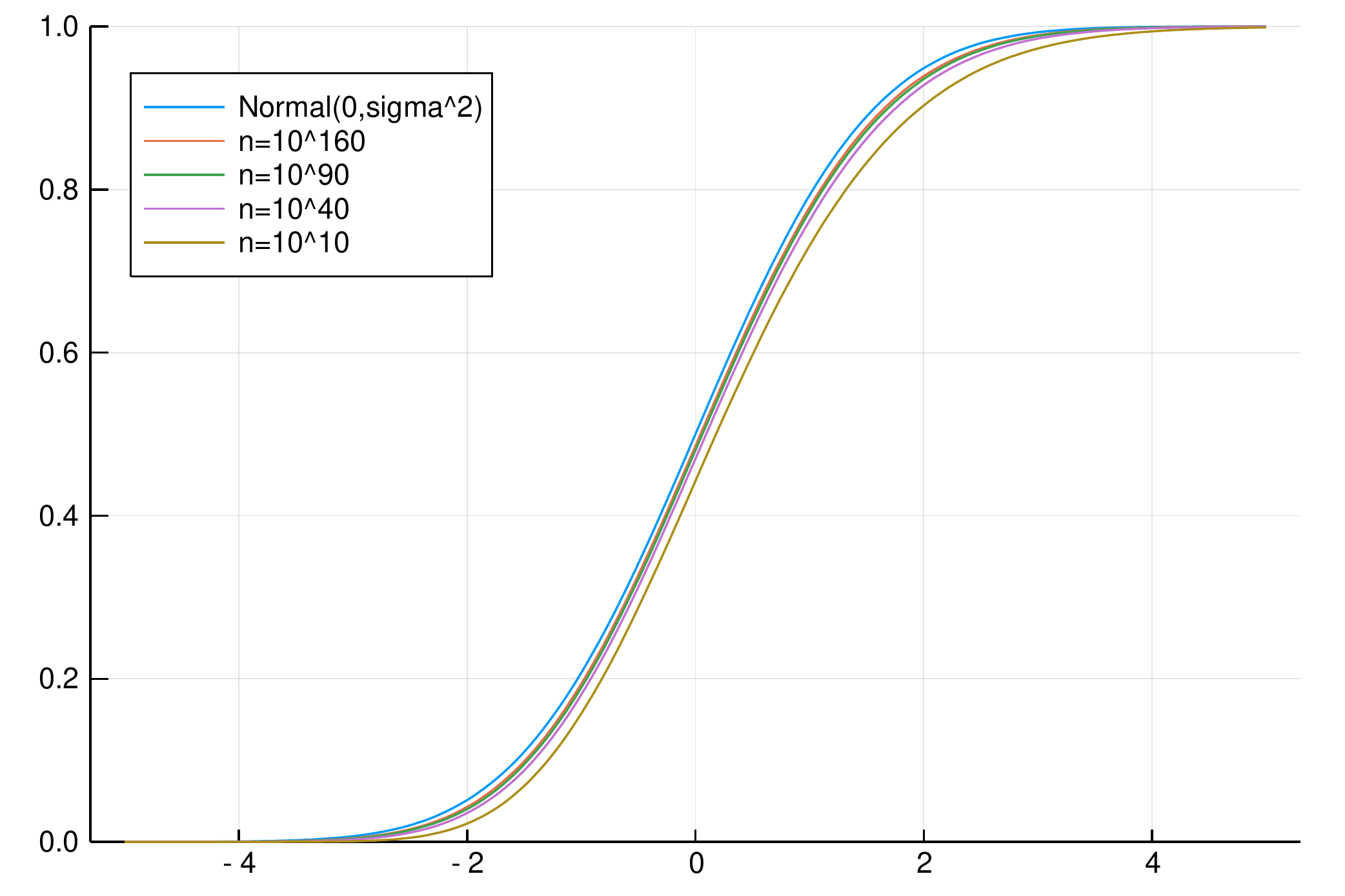}
		\caption{CPP($\lambda=1$) with Pareto($\alpha=4$) steps}
		\label{subfig:alpha195}
	\end{subfigure}
	\label{subfigi}

	\caption{Empirical cumulative distribution functions when the L\'evy subordinator $S$ is a CPP($\lambda=1$) with Pareto($\alpha$) steps}
	\label{fig}
\end{figure}

To focus on testing the convergence in terms of the index of stability $\alpha$ and $n$, we consider the sequence  $m_n :=  n$ and choose as L\'evy subordinator $S$ a compound Poisson process (CPP) with rate $\lambda=1$ and with Pareto($\alpha$) step sizes.
In~\autoref{fig} we show, for $\alpha=2.5$, $\alpha=4$ and for several large values of $n$, the empirical cumulative distribution function resulting from $10^6$ Monte Carlo simulations of the left- and right-hand sides of~\eqref{theo1 lim1}.

To simulate the left-hand side of~\eqref{theo1 lim1}, we sample the last order statistic of  a multivariate random vector $T$ having a LFMO distribution.
This is done using the property that $T$ is equal in distribution to the joint distribution of the times at which the subordinator $S$ up-crosses the collection of iid exponentially distributed ``triggers'', as defined in  equation~\eqref{eq:Ti}; hence, in particular, $T_{n:n}$ is the time at which the L\'evy subordinator upcrosses the maximum of $n$ standard exponential random variables. 
To simulate this maximum of exponentials, for $n$ smaller than $10^{12}$ we use the CDF inverse method, and for $n$ larger than $10^{12}$ we approximate 
it by the distribution of $G + \log n$, where $G$ is a standard Gumbel random variable.
This approximation is accurate enough for these values of $n$ because the Relative Entropy between these two distributions is of order $\OO(1/n)$.

Our reasoning to choose the parameters we test in our experiments in~\autoref{fig} is as follows.
First, we note that the two main ingredientes in the proof of our result are, first, the convergence of the binomial representation of the order statistics of the LFMO distribution, and second, the convergence to a stable distribution of a centered and scaled L\'evy subordinator.
Intuitively then, this suggests that the convergence rate in our result will be greatly influenced by this latter convergence.
Note now that a properly centered and scaled CPP with Pareto($\alpha$) stepsizes, with $\alpha>3$, will converge to a normal random variable at the rate dictated by the Berry-Esseen theorem; hence we conjecture that in this regime the convergence rate in our result will be $\OO\left( (\log n)^{-1/2} \right)$.
This motivates us to choose the values $n=10^{10}$, $10^{40}$, $10^{90}$ and $10^{160}$ for the case $\alpha=4$ in~\autoref{fig} part (b), and indeed we see that the empirical distributions behaves very similar to the limit normal distribution for $n$ larger than $10^{10}$.
On the other hand, the hypotheses of the Berry-Esseen theorem are no longer satisfied when $\alpha \leq 3$ in the case of the centered and scaled CPP with Pareto($\alpha$) stepsizes.
In fact, in~\autoref{fig} part~(a) we see that for the case of $\alpha=2.5$ with $n=10^{10}$, $10^{40}$, $10^{90}$ and $10^{160}$ the empirical distribution behaves similar to the limit one on the right-hand side of the plot, however on the left-hand side the convergence seems to be slower.
Moreover, in the case of a centered and scaled L\'evy subordinator converging to a stable random variable with index of stability $\alpha \in (1,2)$, it is known that convergence of the rescaled L\'evy process is given by a regularly varying function and  there is no lower limit for the speed: any slowly varying function tending to zero can serve as the rate function; see~\cite{de1999exact}.
Therefore, techniques to obtain tight bounds on the speed of convergence will depend on the underlying L\'evy process and the speeds could be very slow. 

Finally, we remark that the L\'evy subordinators we use in our computational experiments satisfy \hyperlink{hypB}{Hypothesis~(B)}, so the limit in the right-hand side of~\eqref{theo1 lim1} is a normal random variable.
Nonetheless, for subordinators satisfying instead \hyperlink{hypA}{Hypothesis~(A$_\alpha$)} one could use, e.g., the stable random variable sampler for Julia~v0.6.4 available at~\cite{White2013}, that uses Notation~1 of~\cite[Section 1.3]{nolan2018}.
We remark, however, that there are multiple parameterizations for stable random variables ---see the discussion in~\cite[Section 1.3]{nolan2018}--- so care is needed to produce samples consistent with the parameterization of~\cite[Section 4.5.1]{whitt2002stochastic}, which is the one we use in this paper.

\section{Proof of the main result}\label{sec:proof}

In this section, we present the proof of our main result,~\autoref{theo1}.
For that purpose, in~\autoref{subsec:elements}, we present the main plan and lemmas used in the proof, and in~\autoref{subsec:proof} we give the actual proof.

\subsection{Elements of the proof}\label{subsec:elements}

The plan of the proof for both parts 1.~and 2.~of \autoref{theo1} is the same and is as follows.

We start by considering the sequence of random variables for which we want to show that converge, i.e., the ones on the left-hand sides of~\eqref{theo1 lim1} and~\eqref{theo1 lim2} for parts 1.~and 2., respectively, and write their tail distribution function conditional on the trajectory of the underlying L\'evy subordinator $S$.
The first step of the plan is to show that we can construct a sequence of deterministic real functions, say $(T_n)_n$, and a sequence of real random variables, say $(\xi_n)_n$, such that the conditional tail probability at hand is almost surely equal to $T_n(\xi_n)$.
Here, the random variables $\xi_n$ are related to a certain ``zoom-out'' centering and scaling of the L\'evy subordinator $S$, and the functions $T_n$ are related to the cumulative distribution function of a binomial random variable.

The second step in the proof plan is to show that the sequence $(\xi_n)_n$ converges in distribution.
We argue this by using the Kolmogorov-Gnedenko results that generalize the central limit theorem in the heavy-tailed setting; see, e.g.,~\cite[Section 4.5]{whitt2002stochastic}.

The third step in the plan is to show that the sequence $(f_n)_n$ converges pointwise almost everywhere.
To argue this, we will need to use the following result, whose proof we defer to the appendix.

\begin{lemma}\label{lemma3}
Let $(p_n)_n$ be a sequence in $(0,1)$ and denote by $\B{n}{p_n}$ a binomial random variable with parameters $(n, p_n)$.
Let $(k_n)_n$ be sequence of nonnegative integers.
\begin{enumerate}
	\item If $n p_n \to 0$ as $n \to +\infty$ then
	$\PP \left( \B{n}{p_n} \leq k_n \right) \to 1$.
	
	\item If $n p_n \to +\infty$ and $p_n \to 0$ as $n \to +\infty$ then
	$\PP \left( \B{n}{p_n} \leq k_n \right) \to 0$ and $\lim_n (k_n - n p_n) / \sqrt{n p_n} = -\infty$ are equivalent.
\end{enumerate}
\end{lemma}
Finally, the fourth and last step of the proof plan is to argue that $T_n(\xi_n)$ converges in distribution since $T_n$ and $\xi_n$ converge pointwise and in distribution, respectively.
For this, we will use~\autoref{theo:cmt} below, which is a generalized version of the continuous mapping theorem for weak convergence.

\begin{lemma}[Theorem~4.27 of~\cite{kallenberg2002foundations}]\label{theo:cmt}
Let $\xi$ and $(\xi_n)$ be random variables on a metric space $X$ satisfying $\xi_n \Rightarrow \xi$ as $n \to +\infty$.
Consider another metric space $Y$, and let $\tau$ and $(\tau_n)$ be measurable mappings from $X$ to $Y$.
Assume that for some measurable set $\Xi \subseteq X$, it holds that $\xi \in \Xi$ almost surely and that for all $x \in \Xi$ and all sequences $(x_n)$ in $X$ such that $x_n \to x$ we have that $\tau_n(x_n) \to \tau(x)$. 
Then, it holds that  $\tau_n(\xi_n) \Rightarrow \tau(\xi)$ as $n \to +\infty$.
\end{lemma}

\subsection{Proof of \autoref{theo1}}\label{subsec:proof}

We now prove~\autoref{theo1}.
To clarify the exposition, we present the proofs of parts~1.~and~2. \linebreak separately.

\begin{proof}[Proof of \autoref{theo1}~part 1.]
We will actually show that for all $t \in \R$, we have that
\begin{align}\label{theo1 result0}
\PP \left( \left. T_{m_n:n} > \frac{\log n + t \left(\log n\right)^{1 / \alpha}}{\EE S_1} \right| \condS  \right) \Longrightarrow \I{\sigma \Sigma_\infty + t < 0}
\end{align}
as $n \to +\infty$ if $n-m_n = o\left( (\log n)^{1 / \alpha} \right)$, where $\Sigma_\infty$ is a $\text{Stable}_\alpha (1, 1, 0)$ random variable and $\sigma$ is as defined in~\eqref{def:sigma}.
This result immediately implies part~1.~of \autoref{theo1}.
Indeed, since both random variables on the left- and right-hand sides of display~\eqref{theo1 result0} have bounded support, then the convergence also holds when taking expected value; thus, after rearranging terms we obtain that if $n-m_n = o\left( (\log n)^{1 / \alpha} \right)$ then for all $t$ we have as $n \to +\infty$, 
\begin{align}\label{theo1 result}
\PP \left( \frac{T_{m_n:n} - \log n / \EE S_1}{\left(\log n\right)^{1 / \alpha} / \EE S_1} > t \right) \to \PP( -\sigma \Sigma_\infty > t).
\end{align}
We conclude by noting that $-\sigma \Sigma_\infty$ is distributed as a Stable$_\alpha (\sigma, -1, 0)$ random variable under \hyperlink{hypA}{Hypothesis~(A$_\alpha$)} ---see~\cite[Section 4.5.1]{whitt2002stochastic}--- and as a Normal$(0,\sigma^2)$ random variable under \hyperlink{hypB}{Hypothesis~(B)}, and that the limit~\eqref{theo1 result} holding for all $t$ characterizes the limit in distribution~\eqref{theo1 lim1}.

The first part of the proof consists of showing that for all $t \in \R$ and all $n$ sufficiently large such that $\log n + t (\log n)^{1 / \alpha} \geq 0$, we have that
\begin{align}\label{lim part 0}
\PP \left( \left. T_{m_n:n} > \frac{\log n + t \left(\log n\right)^{1 / \alpha}}{\EE S_1} \right| \condS \right) = 1 - f_n (\sigma \Sigma_n+t)
\end{align}
almost surely.
Here, the (deterministic) functions $(f_n)_n$ are defined as
\begin{align}\label{def:fn}
f_n(x) := \begin{cases}
\PP \left( \B{n}{e^{ -x (\log n)^{1 / \alpha}} / n} \leq n-m_n \right) & \text{if } x \geq -\frac{\log n}{(\log n)^{1 / \alpha}} \\
\PP \left( \B{n}{1} \leq n-m_n \right) & \text{otherwise,}
\end{cases}
\end{align}
with $\B{n}{p}$ denoting a binomial random variable with parameters $n$ and $p$.
The sequence of random variables $(\Sigma_n)_n$ is defined as
\begin{align}\label{def:Sigman}
\Sigma_n & := \left( \frac{S_{u_n} -{u_n} \EE S_1}{\sigma \left( u_n \EE S_1 \right)^{1 / \alpha}} \right) \left( 1 + t (\log n)^{-(\alpha-1)/\alpha} \right)^{1 / \alpha},
\end{align}
with $\sigma$ as in~\eqref{def:sigma} and where $u_n = u_n(t)$ is defined as
\begin{align}\label{def:un}
u_n := \frac{\log n + t (\log n)^{1 / \alpha}}{\EE S_1}.
\end{align}
We remark that, by the definition of $u_n$, the condition that $n$ is sufficiently large such that $\log n + t (\log n)^{1 / \alpha} \geq 0$ is equivalent to the condition $u_n \geq 0$.
We also remark that the condition $u_n \geq 0$ implies that the argument $\sigma \Sigma_n+t$ in the term $f_n (\sigma \Sigma_n+t)$ is always in the domain $x \geq -\log n / (\log n)^{1/\alpha}$ in the definition of $f_n$, i.e., $\sigma \Sigma_n+t \geq -\log n / (\log n)^{1/\alpha}$ almost surely; this is direct by using that $S$ has nondecreasing paths with $S_0=0$ and using the definitions~\eqref{def:sigma}, \eqref{def:Sigman} and~\eqref{def:un} of $\sigma$, $\Sigma_n$ and $u_n$, respectively.

Indeed, to prove the characterization~\eqref{lim part 0} of the term
$$\PP \left( \left. T_{m_n:n} > (\log n + t \left(\log n\right)^{1 / \alpha}) / \EE S_1 \right| \condS \right),$$
we first use that $S$ is a L\'evy subordinator with $S_0 = 0$, and that given $\condS$, we know that the times $T_1, \ldots, T_n$ are iid with the event $\{T_k > u\}$ having probability $e^{-S_u}$, i.e., $\PP(T_k>u | \condS) = e^{-S_u}$.
This implies that for all $u \geq 0$ and all $k = 1, \ldots, n$ we have
\begin{align*}
\PP \left( \left. T_{k:n} > u \right| \condS \right) & = \PP \left( \left. \B{n}{e^{-S_u}} > n-k \right| \condS \right).
\end{align*}
In particular, for $k=m_n$ we have
\begin{align}\label{eq:binom}
\PP \left( \left. T_{m_n:n} > u \right| \condS \right) & = 1 - \PP \left( \left. \B{n}{e^{-S_u}} \leq n-m_n \right| \condS \right).
\end{align}
Now, note that 
\begin{align*}
e^{-S_u}
 = \exp \left(-\sigma \left( \frac{S_u -u \EE S_1}{\sigma \left(u \EE S_1 \right)^{1 / \alpha}} \right) (u \EE S_1)^{1 / \alpha}+ \log n - u \EE S_1 \right) / n.
\end{align*}
In particular, plugging-in $u = u_n$, with $u_n$ as defined in~\eqref{def:un} and $n$ sufficiently large such that $u_n \geq 0$, we obtain that
\begin{align*}
\lefteqn{ e^{-S_{u_n}} = \exp \left( -\sigma \left( \frac{S_{u_n} - {u_n} \EE S_1}{\sigma \left( u_n \EE S_1 \right)^{1 / \alpha}} \right) \left( \log n + t (\log n)^{1 / \alpha} \right)^{1 / \alpha} - t (\log n)^{1 / \alpha} \right) / n } \\
& = \exp \left( -\left[ \sigma \left( \frac{S_{u_n} -{u_n} \EE S_1}{\sigma \left(u_n \EE S_1 \right)^{1 / \alpha}} \right) \left( 1 + t (\log n)^{-(\alpha-1)/\alpha} \right)^{1 / \alpha} + t \right] (\log n)^{1 / \alpha} \right) / n \\
& = \exp \left( -\left( \sigma \Sigma_n + t \right) (\log n)^{1 / \alpha} \right) / n,
\end{align*}
where in the last equation, we used the definition~\eqref{def:Sigman} of $\Sigma_n$.
Plugging in this expression for $e^{-S_{u_n}}$ and $u = u_n = (\log n + t \left(\log n\right)^{1 / \alpha}) / \EE S_1$ in equation~\eqref{eq:binom}, we obtain that, almost surely,
\begin{align*}
\lefteqn{\PP \left( \left. T_{m_n:n} > \frac{\log n + t \left(\log n\right)^{1 / \alpha}}{\EE S_1} \right| \condS \right) } \\
& = 1 - \PP \left( \left. \B{n}{e^{ -(\sigma \Sigma_n+t) (\log n)^{1 / \alpha}} / n } \leq n-m_n \right| \condS \right) \\
& = 1 - f_n (\sigma \Sigma_n + t),
\end{align*}
where in the last equality we used the definition~\eqref{def:fn} of $f_n$, and that for all $n$ sufficiently large such that $u_n \geq 0$ we have that $\Sigma_n$ is measurable with respect to the sigma-algebra generated by $\condS$.
This proves equation~\eqref{lim part 0}.

The second part of the proof consists of showing that the sequence $(\Sigma_n : n \geq 1)$, with $\Sigma_n$ as defined in~\eqref{def:Sigman}, satisfies the convergence in distribution
\begin{align}\label{lim Sigma_n}
\Sigma_n \Longrightarrow \Sigma_\infty
\end{align}
as $n \to +\infty$, where $\Sigma_\infty$ is a $\text{Stable}_\alpha (1, 1, 0)$ random variable under \hyperlink{hypA}{Hypothesis~(A$_\alpha$)} and a $\text{Normal}(0,1)$ random variable under \hyperlink{hypB}{Hypothesis~(B)}.
Indeed, first note that the term $\left( 1 + t (\log n)^{-(\alpha-1)/\alpha} \right)^{1 / \alpha}$ in the definition of $\Sigma_n$ converges to $1$ as $n \to +\infty$ since $\alpha$ in particular satisfies $\alpha>1$.
On the other hand, using the definition~\eqref{def:sigma} of $\sigma$, we have that as $u \to +\infty$, the random variable $(S_u -u \EE S_1)/(\sigma \left( u \EE S_1 \right)^{1 / \alpha})$ converges in distribution to a $\text{Stable}_\alpha (1, 1, 0)$ random variable under \hyperlink{hypA}{Hypothesis~(A$_\alpha$)} and a $\text{Normal}(0,1)$ random variable under \hyperlink{hypB}{Hypothesis~(B)}; this holds by~\cite[Theorem~4.5.2]{whitt2002stochastic} in the case of \hyperlink{hypA}{Hypothesis~(A$_\alpha$)} and by the central limit theorem in the case of \hyperlink{hypB}{Hypothesis~(B)}.
We conclude the convergence~\eqref{lim Sigma_n} by noting that $u_n \to +\infty$ as $n \to +\infty$.

The third part of the proof consists of showing that the functions $(f_n)_n$ defined in~\eqref{def:fn} satisfy
\begin{align}\label{propfn}
\lim_n f_n(x) = \I{x>0} \text{ for all } x \neq 0 \quad \text{ if and only if } \quad n-m_n = o\left( (\log n)^{1 / \alpha} \right).
\end{align}
Indeed, first note that $\lim_n f_n(x) = 1$ holds for all $x>0$.
This comes from applying part~1.~of \autoref{lemma3} with $p_n = q_n(x)$ defined as
\begin{align}\label{def:pn}
q_n(x) := e^{ -x (\log n)^{1 / \alpha}} / n,
\end{align}
since $n q_n (x) \to 0$ as $n \to +\infty$ for all $x>0$.
Next, we argue that  $\lim_n f_n(x) = 0$ holds for all $x<0$ if and only if $n-m_n = o\left( (\log n)^{1 / \alpha} \right)$.
To do this, we first apply part~2.~of \autoref{lemma3} with $p_n = q_n(x)$ as in~\eqref{def:pn}, since for all $x<0$, we have that $n q_n(x) \to +\infty$ and $q_n(x) \to 0$ since $\alpha>1$; in this way, we obtain that $\lim_n f_n(x) = 0$ holds for all $x<0$ if and only if
\begin{align}\label{statement1}
\lim_n (n-m_n - n q_n(x)) / \sqrt{n q_n(x)} = -\infty \quad \text{for all } x<0.
\end{align}
We now argue that the condition~\eqref{statement1} is equivalent to $n-m_n = o \left( (\log n)^{1 / \alpha} \right)$ as $n \to +\infty$.
For this, recall that $x<0$ and note that rewriting
\begin{align}
\frac{n-m_n - n q_n(x) }{ \sqrt{n q_n(x)}}
& = \frac{n-m_n - e^{ -x (\log n)^{1 / \alpha}}}{e^{ -\frac{x}{2} (\log n)^{1 / \alpha}}} \nonumber \\
& = e^{ \frac{|x|}{2} (\log n)^{1 / \alpha}} \left( e^{(\log n)^{1 / \alpha} \left( \frac{\log (n-m_n)}{(\log n)^{1 / \alpha}} - |x| \right)} - 1 \right) \label{theo1 cond}
\end{align}
we can see that if $(n-m_n) / (\log n)^{1 / \alpha} \to 0$ then~\eqref{statement1} holds; and on the other hand, if $\limsup_n \log (n-m_n) / (\log n)^{1 / \alpha} = \epsilon>0$, then it is sufficient to plug in $x = -\epsilon / 2$ in~\eqref{theo1 cond} to obtain that $\limsup_n (n-m_n - n q_n(x)) / \sqrt{n q_n(x)} = +\infty$.
This proves equivalence~\eqref{propfn}.

The fourth part of the proof consists of showing that
\begin{align}\label{lim cmt}
f_n \left( \sigma \Sigma_n+t \right) \Rightarrow \I{\sigma \Sigma_\infty + t>0}
\end{align}
holds as $n \to +\infty$ when $n-m_n = o \left( (\log n)^{1 / \alpha} \right)$.
Intuitively, this should be true because from the second part of the proof we have $\sigma \Sigma_n+t \Rightarrow \sigma \Sigma_\infty + t$ as $n \to +\infty$, and from the third part of the proof we have that $f_n(x) \to \I{x>0}$ holds for all $x \neq 0$ if $n-m_n = o \left( (\log n)^{1 / \alpha} \right)$.
To formalize this intuition, we apply~\autoref{theo:cmt}, which is a generalized version of the continuous mapping theorem for weak convergence.
Indeed, we can apply this result because, first, $f_n$ is continuous for all $n$; and second, assuming $n-m_n = o \left( (\log n)^{1 / \alpha} \right)$ and defining $\Xi := \R \setminus \{0\}$, we have $\PP (\sigma \Sigma_\infty+t \in \Xi) = 1$ and $\lim_n f_n (x) = \I{x>0}$ for all $x \in \Xi$.
This gives the limit in distribution~\eqref{lim cmt} under the aforementioned growth condition on $n-m_n$.

Finally, using the limit~\eqref{lim cmt} in equation~\eqref{lim part 0}, we obtain that~\eqref{theo1 result0} holds if $n-m_n = o\left( (\log n)^{1 / \alpha} \right)$, which is what we wanted to prove, as we argued in the beginning of this proof.
This concludes the proof of part~1.~of \autoref{theo1}.
\end{proof}

We now give the proof of part~2.~of~\autoref{theo1}.
We give a concise version of the proof, as the arguments parallel the main ideas of the proof of part~1.

\begin{proof}[Proof of \autoref{theo1}~part 2.]
We will actually show that for all $t \in \R$, we have that
\begin{align}\label{theo1 result0 p2}
\PP \left( \left. T_{m_n:n} > t \left(\log n\right)^{1 / \alpha} \right| \condS  \right) \Longrightarrow \I{\sigma \Pi_\infty t^{1/\alpha} < 1}
\end{align}
as $n \to +\infty$ if $n-m_n = o\left( n^\rho \right)$ for all $\rho \in (0,1)$, where $\Pi_\infty$ is a $\text{Stable}_\alpha (1, 1, 0)$ random variable and $\sigma$ is defined as
\begin{align}\label{def:sigma p2}
\sigma := (A / C_\alpha)^{1 / \alpha}.
\end{align}
The result~\eqref{theo1 result0 p2} immediately implies part~2.~of \autoref{theo1}.
Indeed, both random variables on the left- and right-hand sides of display~\eqref{theo1 result0 p2} have bounded support, so the convergence also holds when taking the expected value, thus obtaining after rearranging terms
\begin{align}\label{theo1 result p2}
\PP \left( \frac{T_{m_n:n}}{\left(\log n\right)^{1 / \alpha}} > t \right) \rightarrow \PP\left( \frac{1}{\left( \sigma \Pi_\infty \right)^\alpha} > t \right)
\end{align}
as $n \to +\infty$.
We conclude by noting that $\sigma \Pi_\infty$ has a Stable$_\alpha (\sigma, 1, 0)$ distribution ---see~\cite[Section 4.5.1]{whitt2002stochastic}--- and that the limit~\eqref{theo1 result p2} holding for all $t \in \R$ characterizes the limit in distribution~\eqref{theo1 lim2}, which is what we want to prove.

The first part of the proof consists of showing that for all $t \in \R$ and all $n \geq 1$, we have that 
\begin{align}\label{lim part 0 p2}
\PP \left( \left. T_{m_n:n} > t \left(\log n\right)^{1 / \alpha} \right| \condS \right) = 1 - g_n (\sigma \Pi_n t^{1/\alpha}),
\end{align}
almost surely.
Here, the (deterministic) functions $(g_n)_n$ are defined as
\begin{align}\label{def:fn p2}
g_n(x) := \begin{cases}
\PP \left( \B{n}{1 / n^x} \leq n-m_n \right)  & \text{if } x \geq 0 \\
\PP \left( \B{n}{1} \leq n-m_n \right) & \text{otherwise,}
\end{cases}
\end{align}
and the sequence of random variables $(\Pi_n)_n$ is defined as
\begin{align}
\Pi_n := \frac{S_{v_n}}{\sigma v_n^{1 / \alpha}},
\end{align}
with $\sigma$ as defined in~\eqref{def:sigma p2} and
\begin{align}
v_n = v_n(t):= (\log n)^\alpha t.
\end{align}
Indeed, substituting $u:=v_n$ into equation~\eqref{eq:binom} we obtain
\begin{align*}
\PP \left( \left. T_{m_n:n} > (\log n)^\alpha t \right| \condS \right) & = 1 - \PP \left( \left. \B{n}{e^{-S_{v_n}}} \leq n-m_n \right| \condS \right),
\end{align*}
and rewriting $e^{-S_{v_n}}$ as
\begin{align*}
e^{-S_{v_n}}
= \exp \left( -\sigma \frac{S_{v_n}}{\sigma v_n^{1 / \alpha}} t^{1 / \alpha} \log n \right)
= \exp \left( -\sigma \Pi_n t^{1 / \alpha} \log n \right)
= 1 / n^{\sigma \Pi_n t^{1 / \alpha} }
\end{align*}
we obtain the characterization~\eqref{lim part 0 p2} of $\PP \left( \left. T_{m_n:n} > t \left(\log n\right)^{1 / \alpha} \right| \condS \right)$.

The second part of the proof is to note that by~\cite[Theorem~4.5.2]{whitt2002stochastic}, the random variable $\Pi_n$ converges in distribution to a $\text{Stable}_\alpha (1, 1, 0)$ random variable, say $\Pi_\infty$.

The third part of the proof consists of showing that
\begin{align*}
\lim_n g_n(x) = \I{x>1} \text{ for all } x \in \R_+ \setminus \{0,1\}
\end{align*}
is equivalent to
\begin{align*}
n-m_n = o \left( n^{\rho} \right) \text{ for all } \rho \in (0,1).
\end{align*}
Indeed, using $p_n:= 1/n^x$ in part~1.~of \autoref{lemma3}, we obtain that $\lim_n g_n(x) = 1$ for all $x>1$.
Additionally, applying part~2.~of \autoref{lemma3} with $p_n$ as before, we obtain that $\lim_n g_n(x) = 0$ holds for all $x \in (0,1)$ if and only if $\lim_n (n-m_n-n/n^x) / \sqrt{n/n^x} = -\infty$ holds for all $x \in (0,1)$, which in turn is equivalent to $n-m_n = o \left( n^{\rho} \right)$ for all $\rho \in (0,1)$.
The latter equivalence is checked by noting that
\begin{align*}
\frac{n-m_n-n/n^x}{\sqrt{n/n^x}} &= n^{\frac{1-x}{2}} \left( \frac{n-m_n}{n^{1-x}} - 1 \right),
\end{align*}
so if there exists $\rho \in (0,1)$ such that $\limsup_n (n-m_n)/n^\rho > 0$, then taking $x^* := 1-\rho / 2$, we obtain that $x^* \in (0,1)$ and that $\limsup_n (n-m_n-n/n^{x^*}) / \sqrt{n/n^{x^*}} = +\infty.$

The fourth part of the proof consists of showing that if $n-m_n = o \left( n^{\rho} \right)$ for all $\rho \in (0,1)$, then
\begin{align}\label{lim cmt p2}
g_n \left( \sigma t^{1/\alpha} \Pi_n \right) \Rightarrow \I{\sigma t^{1/\alpha}\Pi_\infty > 1}
\end{align}
as $n \to +\infty$.
This comes from applying~\autoref{theo:cmt} with $\Xi := \R_+ \setminus \{ 0,1 \}$, since from the second part of the proof we have that $\sigma \Pi_n t^{1/\alpha} \Rightarrow \sigma \Pi_\infty t^{1/\alpha}$, with $\PP(\sigma \Pi_\infty t^{1/\alpha} \in \Xi) = 1$; and from the third part of the proof we have that $\lim_n g_n(x) = \I{x>1}$ for all $x \in \Xi$, under the aforementioned growth condition on $n-m_n$.

Finally, using the limit~\eqref{lim cmt p2} in equation~\eqref{lim part 0 p2}, we obtain that~\eqref{theo1 result0 p2} holds if $n-m_n = o\left( n^\rho \right)$ for all $\rho \in (0,1)$, which is what we wanted to prove.
This concludes the proof of part~2.~of \autoref{theo1}.
\end{proof}

\paragraph{Acknowledgements}
Javiera Barrera acknowledges the financial support of FONDECYT grant 1161064 and of Programa Iniciativa Cient\'ifica Milenio NC120062.
Guido Lagos acknowledges the financial support of FONDECYT grant 3180767, Programa Iniciativa Cient\'ifica Milenio NC120062, and the Center for Mathematical Modeling of Universidad de Chile through their grants Proyecto Basal PFB-03 and PIA Fellowship AFB170001.

\begin{appendices}\label{sec:append}

\begin{lemma}\label{lemma1}
Let $(q_n)_n \subset \R$ be a sequence such that $q_n = o(\sqrt n)$ as $n \to +\infty$.
Then, $\left(1+q_n/n\right)^n \sim e^{q_n}$ as $n \to +\infty$, i.e., $\lim_{n \to +\infty} \left(1 + q_n/n\right)^n / e^{q_n} = 1$.
\end{lemma}

\begin{proof}
For all $n$ sufficiently large, we can write
\begin{align*}
\lefteqn{ \frac{\left(1+q_n/n\right)^n}{ e^{q_n}}
= \exp \left( n \log(1+q_n/n) -q_n \right)} \\
& = \exp \left( \sum_{j \geq 2} (-1)^{j+1} j! \left( q_n / \sqrt n \right)^j \frac{1}{n^{\frac{j}{2}-1}} \right) \\
& = \exp \left( -2 \left( q_n / \sqrt n \right)^2 + o \left( \left(q_n / \sqrt n\right)^2 \right) \right),
\end{align*}
where we used the series expansion of the $\log$ function around 1 to write the term $\log(1+q_n/n)$ as a series and where in the last equality, we used that $q_n/n \to 0$ as $n \to +\infty$ because $q_n = o(\sqrt n)$.
We conclude by noting that the last term of the equation goes to one as $n \to +\infty$ since $q_n = o(\sqrt n)$.
\end{proof}

\begin{lemma}\label{lemma2}
Let $(p_n)_n$ be a sequence in $(0,1)$ such that $p_n \to 0$ and $n p_n \to +\infty$ as $n \to +\infty$.
Denoting by $\B{n}{p_n}$ a binomial random variable with parameters $(n, p_n)$, and by $N$ a standard normal random variable, it holds that
\begin{align}\label{lemma2 lim}
\frac{\B{n}{p_n} - n p_n}{\sqrt{n p_n}} \Longrightarrow N
\end{align}
as $n \to +\infty$.
\end{lemma}

\begin{proof}
We will prove that the moment generating function of the random variable in the left of~\eqref{lemma2 lim} converges to the moment generating function of a standard normal random variable.
Indeed,
\begin{align*}
\lefteqn{
\EE \left[ \exp \left( t \frac{\B{n}{p_n} - n p_n}{\sqrt{n p_n}} \right) \right]
= \left( 1-p_n + p_n e^{t / \sqrt{n p_n}} \right)^n e^{-t \sqrt{n p_n}}
} \\
& = \left( 1 + \frac{n p_n \left(e^{t / \sqrt{n p_n}}-1\right)}{n} \right)^n e^{-t \sqrt{n p_n}} \\
& = \left[ \left. \left( 1 + \frac{n p_n \left(e^{t / \sqrt{n p_n}}-1\right)}{n} \right)^n \middle/
\exp \left( n p_n \left(e^{t / \sqrt{n p_n}}-1\right) \right) \right. \right] \\
& \qquad\qquad\qquad\qquad\qquad\qquad \cdot \exp \left( n p_n \left(e^{t / \sqrt{n p_n}}-1\right) - t \sqrt{n p_n} \right).
\end{align*}
Now, by~\autoref{lemma1} we have that the term in the square brackets converges to one, since
$$
\frac{n p_n \left(e^{t / \sqrt{n p_n}}-1\right)}{\sqrt n}
= \frac{n p_n}{\sqrt n} \sum_{j \geq 1} \frac{t^j}{j!} \frac{1}{(n p_n)^{j/2}}
= t \sqrt{p_n} + \frac{1}{\sqrt n} \sum_{j \geq 0} \frac{t^{j+2}}{(j+2)!} \frac{1}{(n p_n)^{j/2}},
$$
which converges to zero as $n \to +\infty$.
Additionally, using the series expansion of the exponential function we obtain that
\begin{align*}
\lefteqn{\exp \left( n p_n \left(e^{t / \sqrt{n p_n}}-1\right) - t \sqrt{n p_n} \right)} \\
&& = \exp \left( \frac{t^2}{2} + \frac{1}{\sqrt{n p_n}} \sum_{j \geq 0} \frac{t^{j+3}}{(j+3)!} \frac{1}{(n p_n)^{j/2}} \right),
\end{align*}
which converges to $\exp(t^2/2)$ since $n p_n \to +\infty$.
This concludes the proof.
\end{proof}

We are now able to prove~\autoref{lemma3}.

\begin{proof}[Proof of~\autoref{lemma3}]
For part~1.~it is enough to show that $\B{n}{p_n} \Rightarrow 0$.
This holds since
$\PP ( \B{n}{p_n} = 0) = \left( 1-p_n \right)^n = \left( 1 - n p_n / n \right)^n \sim e^{-n p_n} \sim 1$
as $n \to +\infty$, where the asymptotic equalities hold due to~\autoref{lemma1} and $n p_n \to 0$.

We now prove part~2.
For that purpose define first $F_n (t) := \PP( \B{n}{p_n} \leq n p_n + \sqrt{n p_n} t)$ and $a_n := (k_n - n p_n) / \sqrt{n p_n}$.
Note that part~2.~of \autoref{lemma3} is equivalent to
\begin{align}\label{lemma3:equiv}
\limsup F_n(a_n) = 0 \qquad \text{if and only if} \qquad \limsup a_n = -\infty.
\end{align}
To prove the reverse implication of~\eqref{lemma3:equiv} note that for all $b$ and all sufficiently large $n$ we have $F_n (a_n) \leq F_n(b)$, since $a_n \leq b$ for all sufficiently large $n$.
Taking then $\limsup$ we obtain by~\autoref{lemma2} that $\limsup F_n(a_n) \leq \Phi(b)$, where $\Phi$ is the cumulative distribution function of the standard normal distribution.
We conclude that $\limsup F_n(a_n) = 0$ by making $b \searrow -\infty$.
Now, to prove the direct implication of~\eqref{lemma3:equiv}, let $\limsup F_n(a_n) = 0$ and assume ad absurdum that $\limsup a_n = c > -\infty$.
Consider then a subsequence $(a_{n_j})_j$ such that $a_{n_j} \nearrow c $ as $j \to +\infty$.
Then for all $\epsilon > 0$ and all sufficiently large $j$ it holds that
\begin{align}\label{lemma3 ineq1}
F_{n_j} (c-\epsilon) \leq F_{n_j} (a_{n_j}) \leq F_{n_j} (c),
\end{align}
since $a_{n_j} \in [c-\epsilon, c]$ for all sufficiently large $j$.
Making $j \to +\infty$ in~\eqref{lemma3 ineq1} we obtain that
\begin{align}\label{lemma3 ineq2}
\Phi (c-\epsilon) \leq \liminf F_{n_j} (a_{n_j}) \leq \limsup F_{n_j} (a_{n_j}) \leq \Phi (c),
\end{align}
where we used that by~\autoref{lemma3} the sequence $(F_{n_j})_j$ converges pointwise to $\Phi$.
It follows that by taking $\epsilon \searrow 0$ in equation~\eqref{lemma3 ineq2} we obtain that $\lim_j F_{n_j} (a_{n_j}) = \Phi(c)$.
Thus, in particular, $\Phi(c)$ is an accumulation point of $(F_n(a_n))_n$, which in turn implies $\Phi(c) \leq \limsup F_n(a_n)$.
But $\Phi(c)>0$ because $c>-\infty$, and we had assumed $\limsup F_n(a_n) = 0$, which is a contradiction.
This shows that necessarily $\limsup a_n = -\infty$.
\end{proof}

\end{appendices}



\bibliographystyle{plain}
\bibliography{bibliography}   


\end{document}